\documentclass{amsart}

\usepackage[utf8]{inputenc}
\usepackage{amsmath,amssymb,amsthm,bbm,centernot,xcolor,tikz,upgreek,esint,float,comment,mathrsfs,setspace}
\usepackage{hyperref}
\usepackage[shortlabels]{enumitem}

\renewcommand{\gamma}{\upgamma}

\newcommand{\comm}[1]{\ignorespaces}

\newtheorem*{Thm*}{Theorem}

\newtheorem*{Prop*}{Proposition}

\newtheorem*{Lem*}{Lemma}
\newtheorem{thm}{Theorem}[section]
\newtheorem*{thm*}{Theorem}
\newtheorem{prop}[thm]{Proposition}
\newtheorem{remark}[thm]{Remark}
\newtheorem*{prop*}{Proposition}
\newtheorem{cor}[thm]{Corollary}
\newtheorem{lem}[thm]{Lemma}
\newtheorem*{lem*}{Lemma}
\newtheorem*{Cla*}{Claim}
\theoremstyle{remark}

\theoremstyle{definition}
\newtheorem{Def}[thm]{Definition}
\newtheorem*{Def*}{Definition}

\newtheorem*{rem*}{Remark}

\numberwithin{equation}{section}

\title{On the equivalence of derivatives for maps between Carnot groups}

\author{Scott Zimmerman}

\date{\today}

\address{Department of Mathematics, The Ohio State University, 100 Math Tower (MW) 231 W 18th Ave, Columbus, OH}
\email{zimmerman.416@osu.edu}

\begin{document}
\maketitle

\begin{abstract}
This paper gives an alternate, elementary proof of a result of Magnani: maps between Carnot groups that preserve horizontal curves and are continuously differential in horizontal directions in the Euclidean sense are continuously Pansu differentiable. This proof contains primarily Euclidean arguments and also reproves a version of Magnani's mean value estimate for continuously Pansu differentiable maps.
\end{abstract}


\section{Introduction}

Suppose $\mathbb{G}$ and $\hat{\mathbb{G}}$ are Carnot groups.
Suppose also that $\Omega \subset \mathbb{G}$ is open and $f:\Omega \to \hat{\mathbb{G}}$ is a map that preserves horizontal curves.
Pansu introduced a natural notion of differentiability for maps between Carnot groups in \cite{Pansu}.
Since Carnot groups can be viewed as Euclidean spaces with a particular Lie group structure, it also makes sense to discuss the classical Euclidean differentiability of maps between Carnot groups.
These notions of differentiability are not the same.
In particular, there are continuously Pansu differentiable maps between Carnot groups that are not Euclidean differentiable in non-horizontal directions \cite{MagCoarea}. 

When $\Omega = \mathbb{G} = \hat{\mathbb{G}}$, Warhurst \cite{Warhurst} proved that $f$ is continuously Pansu differentiable if and only if it is a Euclidean-$C^1$ diffeomorphism.
In the more general setting without these added restrictions, Magnani proved that 
$f$ is continuously Pansu differentiable if and only if it is 
continuously Euclidean differentiable in the horizontal directions of $\hat{\mathbb{G}}$. See Theorem~1.1 in \cite{MagnaniToward}.
The forward implication of this result is straightforward.
It is the reverse implication that is the focus of this paper;
 we provide an alternate, elementary proof.
Magnani argued at the level of Lie algebras, while our proof below will be mostly Euclidean in flavor.

Magnani also provided a mean value estimate on the Pansu difference quotients of continuously Pansu differentiable maps. See Theorem~1.2 in \cite{MagnaniToward}. In order to prove our main goal, we also establish Magnani's estimate (with a weaker exponent on the modulus of continuity).

Before stating the theorem we aim to prove, let us first establish some notation.
For any open set $\Omega \subset \mathbb{G}$ and compact $A \subset \Omega$, define $t_A = \frac{\text{dist}(A, \mathbb{G} \setminus \Omega)}{2k_0 C_0\sup_{ \Vert \xi \Vert = 1} |\xi|}$ where $k_0$ and $C_0$ are the constants depending only on $\mathbb{G}$ from Remark~\ref{r-strat}, and set $\mathcal{A} = \{ y \in \mathbb{G} : \text{dist}(y,A) \leq \tfrac12 \text{dist} (A,\mathbb{G} \setminus \Omega)\}$.
In the following, we will use $\hat{d}$ and $\hat{n}$ to respectively denote the metric on and horizontal dimension of $\hat{\mathbb{G}}$.

\begin{thm}
    \label{t-actualmain}
    Suppose $\Omega \subset \mathbb{G}$ is open. If $f: \Omega \to \hat{\mathbb{G}}$ preserves horizontal curves and $f_i$ is Euclidean-$C^1$ on $\Omega$ for $i \leq \hat{n}$
    then,
    for any compact set $A \subset \Omega$, there is a constant $C > 0$ satisfying the following:
    for every $x \in A$ and $\xi \in \partial B(0,1)$, there is some $z(x,\xi) \in \hat{\mathbb{G}}$ such that
    \begin{equation}
        \label{e-main1}
        \hat{d}(R(x,\xi;t),z(x,\xi)) \leq C\omega(t)^{1/s^{k_0}}
    \end{equation}
    for $t \in \left(0,t_A\right]$
    where $\omega$ is a Euclidean modulus of continuity on $\mathcal{A}$ for each $\nabla f_i$ with $i \leq \hat{n}$.
        If we further write $\xi = \xi_{k_0} \cdots \xi_1$ as in Remark~\ref{r-strat}, then we have
    \begin{equation}
        \label{e-main2}
    z(x,\xi) = D_{\xi_{k_0}}f(x) D_{\xi_{k_0 -1}}f(x) \cdots D_{\xi_1}f(x),
    \end{equation} 
    where
\begin{align}
D_{\xi_i} f(x)
&=
\left( \left. \tfrac{d}{dt} f_1(x \delta_t(\xi_i)) \right|_{t=0},\dots,\left. \tfrac{d}{dt} f_{\hat{n}}(x \delta_t(\xi_i)) \right|_{t=0},0,\dots,0 \right)
\nonumber
\\
&=\lambda_i \left( \nabla f_1(x) \cdot X_{j_i}(x), \dots, \nabla f_{\hat{n}}(x) \cdot X_{j_i}(x),0,\dots,0 \right).\label{e-otherform}
\end{align}
\end{thm}

The bound of $t_A$ exists to ensure that $R(x,\xi;t)$ is well-defined and well controlled (i.e. $x \delta_t (\xi) \in \mathcal{A}$).
This result can be interpreted as follows. Inequality \eqref{e-main1} implies that $f$ is Pansu differentiable and that its Pansu difference quotients converge uniformly at a rate controlled by the uniform convergence of the Euclidean derivative. This is the special case of the mean value estimate (Theorem~1.2 in \cite{MagnaniToward}) mentioned above. Moreover, \eqref{e-main2} and \eqref{e-otherform} give an explicit expression of the Pansu derivative in this setting.
While this formulation is likely not new, a relevant reference was not readily found.
Since each $D_{\xi_i}f$ is continuous, we may conclude the following.

\begin{cor}[\cite{MagnaniToward}]
        Suppose $\Omega \subset \mathbb{G}$ is open and $f: \Omega \to \hat{\mathbb{G}}$ preserves horizontal curves.
    If $f_i$ is Euclidean-$C^1$ on $\Omega$ for $i \leq \hat{n}$, then $f$ is continuously Pansu differentiable.
\end{cor}

The main new contribution in this paper and the tool underpinning the proof of Theorem~\ref{e-main1} is Lemma~\ref{l-approx}.
This lemma shows that smooth horizontal curves are well approximated by ``horizontal rays'' in the same way that Euclidean curves are well approximated by tangent lines. 

The paper is outlined as follows.
Preliminary information about Carnot groups and moduli of continuity are given in Section~\ref{s2}.
Section~\ref{s3} is devoted to the approximation lemma discussed in the previous paragraph. The proof of Theorem~\ref{e-main1} is the primary content of Section~\ref{s4}.

\section{Preliminaries} \label{s2}

\subsection{Carnot groups}
Given a positive integer $s$, suppose $\mathbb{G}$ is
a connected, simply connected Lie group 
and $V_1,\dots, V_s$ are non-zero subspaces of its associated Lie algebra $\mathfrak{g}$.
We say that $\mathbb{G}$
is
a \textit{step-$s$ Carnot group} if $\mathfrak{g}$ is stratified in the following
sense:
$$
\mathfrak{g} = V_1 \oplus \cdots \oplus V_s, \quad
[V_1,V_i]=V_{i+1} \text{ for } i=1,\dots, s-1, 
\quad 
[V_1,V_s] = \{0\}.
$$
The \textit{horizontal subbundle} $H$ of $T \mathbb{G}$ is defined so that $H_0=V_1$, and, for any $p \in \mathbb{G}$, we set $H_p = (L_p)_* (H_0)$ where $L_p(x) = px$ is the left translation operation. Set $N := \text{dim } \mathfrak{g}$ and $n := \text{dim } V_1$.
We call $n$ the {\em horizontal dimension} of $\mathbb{G}$.

We select a basis $\{X_1,\dots, X_N\}$ of $\mathfrak{g}$ adapted to the stratification in the sense that
$
\{X_{(\sum^{i-1}_{j=1} \text{dim } V_j)+1} , \dots , X_{\sum^{i}_{j=1} \text{dim } V_j}\}$
 forms a basis of $V_i$ for each $i \in \{1,\dots,s\}$.
For any $x \in \mathbb{G}$, then, we can uniquely
write $x = \text{exp}(x_1X_1 + \cdots + x_NX_N)$ for some $(x_1,\dots, x_N) \in \mathbb{R}^N$ via the exponential map $\text{exp} : \mathfrak{g} \to \mathbb{G}$. This allows us to identify $\mathbb{G}$ with $\mathbb{R}^N$ via the relationship
$x \leftrightarrow (x_1,\dots, x_N)$. 
Denote by $|\cdot|$ the Euclidean norm in $\mathbb{G} = \mathbb{R}^N$ determined by our choice of basis above.
Note also that, under this identification, we have $x^{-1} = -x$ for any $x \in \mathbb{G}$.

Another property that follows from the stratified structure of a Carnot group is the existence of automorphic dilations.
Note that, for any coordinate $i \in \{ 1,\dots,N\}$, there is a positive integer $d_i$ satisfying
$$
\sum^{d_i-1}_{j=1} \text{dim } V_j < i \leq \sum^{d_i}_{j=1} \text{dim } V_j.
$$
We call $d_i$ the {\em degree} of $i$. Intuitively, this describes the ``layer'' of the stratification in which the coordinate $i$ lives.
Note in particular that, if $i \leq n$, then $d_i = 1$.
For any $t > 0$, define the {\em dilation} $\delta_t$ as
$$
\delta_t(x) = (t^{d_1}x_1,t^{d_2}x_2,\dots,t^{d_N}x_N).
$$
This forms a one parameter family of automorphisms in the sense that $\delta_{t_1} \circ \delta_{t_2} = \delta_{t_1 t_2}$ for $t_1,t_2 > 0$.
We can naturally extend this one parameter family to $t=0$ so that $\delta_0(x) = 0$, but this is clearly no longer an automorphism.

The following well known fact follows from the stratified structure of $\mathbb{G}$. See, for example, Lemma~1.40 in \cite{FolSteinBook}.
This fact allows us to ``build'' any point in $\mathbb{G}$ out of boundedly many points from the first layer.

\begin{remark}
    \label{r-strat}
    There is a positive integer $k_0$ and a constant $C_0>0$ such that, 
    for each $\xi \in \mathbb{G}$, we may write $\xi = \xi_{k_0} \cdots \xi_1$ where, for $k = 1,\dots,k_0$, we have 
    $\xi_k = \lambda_k e_{j_k}$ for some $j_k \in \{1,\dots, n\}$ and $\lambda_k \in \mathbb{R}$ with $|\lambda_k| \leq C_0 |\xi|$.
\end{remark}

Since the Lie algebra is nilpotent, one may explicitly compute the group operation in $\mathbb{G}$ from the Lie bracket combinations in $\mathfrak{g}$ via the Baker–Campbell–Hausdorff formula (see Definition~2.2.11 in \cite{ItalianBook}).
The following properties of the group law in $\mathbb{G}$ follow from this computation and are well known. See, for example, Proposition~2.2.22 in \cite{ItalianBook}.
\begin{prop}
    \label{p-structure}
    We may write the group law as $xy = x+y+Q(x,y)$ for some polynomial $Q=(Q_1,\dots,Q_N)$ where
    \begin{enumerate}
        \item $Q_1=\dots=Q_n=0$,
        \item for $n < i \leq N$, the real valued polynomial $Q_i(x,y)$ may be written as a sum of terms each of which contains a factor of the form $(x_jy_\ell - x_\ell y_j)$ for some $1 \leq \ell < i$ and $1 \leq j < i$,
        \item  $Q_i$ is homogeneous of degree $d_i$ with respect to dilations i.e. $Q_i(\delta_t (x),\delta_t(y)) = t^{d_i} Q_i(x,y)$ for all $x,y \in \mathbb{G}$ and $t > 0$,
        \item if the coordinate $x_i$ has degree $d_i \geq 2$, then $Q_i(x, y)$ depends only on the coordinates of $x$ and $y$ which have degree strictly less than $d_i$.
    \end{enumerate}
\end{prop}
Note in particular that the group operation is simply linear in the first layer and, in any higher layer, depends only on the previous ones.
This is expected due to the stratified structure of $\mathfrak{g}$.

A metric $d$ on $\mathbb{G}$ is called \textit{left invariant} if $d(xp,xq) = d(p,q)$  and is called \textit{homogeneous} if $d(\delta_t(p),\delta_t(q)) = t d(p,q)$ for $x,p,q \in \mathbb{Q}$ and $t>0$.
All homogeneous, left invariant metrics on $\mathbb{G}$ are bi-Lipschitz equivalent to one another due to the compactness of their unit spheres.
For any such metric we also have the following H\"{o}lder-type equivalence (and thus topological equivalence) with the Euclidean norm on $\mathbb{R}^N$.
See Proposition~5.15.1 in \cite{ItalianBook}.

\begin{prop}
\label{p-compact}
    For any compact set $K \subset \mathbb{G}$, there is a constant $C_K \geq 1$ such that
    $$
    C_K^{-1} |x-y| \leq d(x,y) \leq C_K |x-y|^{1/s}.
    $$
\end{prop}

We will choose a particular left invariant, homogeneous metric $d$ on $\mathbb{G}$ for our purposes below.
For $x \in \mathbb{G}$, set
$$
\Vert x \Vert := \max_{i=1,\dots,N} \mu_i |x_i|^{1/d_i}
$$
where the constants $\mu_i > 0$ are chosen so that $\Vert \cdot \Vert$ satisfies the triangle inequality. Such a choice is always possible; see \cite{Guivarch}.
Define $d(x,y) = \Vert x^{-1}y \Vert$ for $x,y \in \mathbb{G}$.
This metric is left invariant and homogeneous by its definition.

Suppose $\gamma:[a,b] \to \mathbb{G}$ is absolutely continuous as a curve in $\mathbb{R}^N$.
We say that $\gamma$ is {\em horizontal}
if $\gamma'(t) \in H_{\gamma(t)}$ for almost every $t \in [a,b]$.
In particular, this means
    \begin{equation}
        \label{e-horizniceintro}
    \gamma'(t) = 
    \sum_{j=1}^n  \gamma_j'(t) X_j(\gamma(t))
    =
    \sum_{j=1}^n \gamma_j'(t) e_j  + 
    \sum_{j=1}^n \gamma_j'(t) \left. \frac{\partial Q(\gamma(t),y)}{\partial y_j}\right|_{y=0}.
    \end{equation}

Suppose $\Omega \subset \mathbb{G}$ is open.
A map $f: \Omega \to \hat{\mathbb{G}}$ is said to {\em preserve horizontal curves} if $f \circ \gamma$ is a horizontal curve in $\hat{\mathbb{G}}$ for every horizontal curve $\gamma$ in $\Omega$.

\subsection{Bounding right translation}
While the metric $d$ is not invariant under right multiplication, we do have the following estimate.
\begin{prop}
    \label{p-rightmult}
    Suppose $M > 0$ and $a,b,c \in \mathbb{G}$ satisfy $\max\{|a|,|b|,|c|\} \leq M$. 
    There is a constant $C>0$ depending only on $M$ and $\mathbb{G}$ so that
    $$
    d(ac,bc) \leq C d(a,b)^{1/s}.
    $$
\end{prop}
\begin{proof}
In this proof and the next, we will write $x \lesssim y$ for quantities $x,y \geq 0$ to indicate that there is a constant $C \geq 1$ depending only on $M$ and $\mathbb{G}$ such that $x \leq Cy$.

    Write $p=a^{-1}b$.
    We have that 
    \begin{align}
        \label{e-rightmult1}
        d(ac,bc) = \Vert c^{-1}pc \Vert = \max_{1\leq i \leq N} \mu_i \left| (c^{-1}pc)_i \right|^{1/d_i}.
    \end{align}
    Let us establish the structure of $c^{-1}pc$.
    First, note that
    \begin{align*}
        pc = \sum_{i=1}^n (p_i + c_i)e_i + \sum_{i=n+1}^N \left(p_i + c_i + Q_i(p,c) \right) e_i
    \end{align*}
    so that
    \begin{align*}
        c^{-1}pc = \sum_{i=1}^n p_i e_i + \sum_{i=n+1}^N \left(p_i + Q_i(p,c) + Q_i(c^{-1},pc) \right) e_i.
    \end{align*}
    According to Proposition~\ref{p-structure}{\em (2)}, every polynomial $Q_i(p,c)$ can be written as a sum of terms
    each of which contains a factor of the form $(p_j c_\ell - p_\ell c_j)$ for some $1 \leq j,\ell < i$.
    In particular, $Q_i(p,c)$ can be rewritten as a sum of terms each of which has a factor of $p_j$ for some $1 \leq j < i$.

    Similarly, we can write $Q_i(c^{-1},pc)$ as a sum of terms each of which contains a factor of the form 
    \begin{align*}
        -c_j (pc)_\ell + c_\ell (pc)_j
        &= 
        -c_j \left(p_\ell + c_\ell + Q_\ell (p,c)\right) + c_\ell \left(p_j + c_j + Q_j (p,c)\right)\\
        &=
        -c_j p_\ell -c_j Q_\ell (p,c) + c_\ell p_j + c_\ell Q_j (p,c).
    \end{align*}
    Again, then, $Q_i(c^{-1},pc)$ can be rewritten as a sum of terms each of which has a factor of $p_j$ for some $1 \leq j < i$.
    Therefore, we may conclude that
    \begin{align}
        \label{e-rightmult2} \left| (c^{-1}pc)_i \right|^{1/d_i} \lesssim \max_{1 \leq j \leq i} \left(|p_j|^{1/d_j}\right)^{d_j/d_i} 
        &\lesssim
        \Vert p \Vert^{1/s} 
    \end{align}
    since $d_j/d_i \geq 1/s$.
    Combining \eqref{e-rightmult1} and \eqref{e-rightmult2} completes the proof.
\end{proof}

\begin{cor}
\label{c-nice}
    Given $\alpha>0$ 
    and $a_1,\dots,a_{k_0},b_1,\dots,b_{k_0} \in \mathbb{G}$ satisfying $d(a_k,b_k) \leq \alpha$ for $k=1,\dots,{k_0}$
    and $\max_{i = 1,\dots,k_0} \{|a_i|,|b_i|\} \leq M$ for some $M > 0$,
    there is a constant $C>0$ depending only on $M$ and $\mathbb{G}$ such that
    $$
    d(a_{k_0}\cdots a_1, b_{k_0} \cdots b_1) \leq C \alpha^{1/s^{k_0 -1}}.
    $$
\end{cor}
\begin{proof}
    Note first that $d(a_1,b_1) \leq \alpha$.
    Suppose for some $k \in \{ 1, \dots, k_0 - 1\}$ that we have proven
    $$
    d(a_{k}\cdots a_1, b_{k} \cdots b_1) \lesssim \alpha^{1/s^{k-1}}.
    $$
    We then have
    \begin{align*}
        d(a_{k+1} &a_k \cdots a_1, b_{k+1} b_k \cdots b_1)\\
        &\leq
        d(a_{k+1}a_k\cdots a_1, a_{k+1}b_k \cdots b_1)
        +
        d(a_{k+1}b_k \cdots b_1, b_{k+1}b_k \cdots b_1)
        =: \mathbf{I} + \mathbf{II}.
    \end{align*}
    The left invariance of the metric provides the bound $\mathbf{I} \lesssim \alpha^{1/s^{k-1}} \lesssim \alpha^{1/s^k}$.
    Applying the previous proposition $k$ times gives
    $
    \mathbf{II} \lesssim d(a_{k+1},b_{k+1})^{1/s^k}
    \lesssim
    \alpha^{1/s^k}.
    $
    Repeating this process until it terminates completes the proof.
\end{proof}

\subsection{Pansu Differentiability}

    Suppose $\Omega \subset \mathbb{G}$ is open,
and consider $f:\Omega \to \hat{\mathbb{G}}$.
The {\em Pansu difference quotient of $f$ at $x \in \mathbb{G}$ in the direction of $\xi \in \partial B(0,1) \subset \mathbb{G}$} is defined for all $t>0$ as
$$
R(x,\xi;t) = \hat{\delta}_{1/t} \left( f(x)^{-1} f\left( x\delta_{t}(\xi) \right) \right).
$$
We say that $f$ is {\em Pansu differentiable at $x \in \Omega$} if there is some $z(x,\xi) \in \mathbb{G}$ 
such that, for all $\xi \in \partial B(0,1)$, the difference quotient $R(x,\xi;t)$ converges to $z(x,\xi)$ as $t \to 0$ and this convergence is uniform with respect to $\xi$.
When this limit exists, the map $\xi \mapsto z(x,\xi)$ is known to be a homeomorphism. See Proposition~3.2 in \cite{Pansu}.
We say that $f$ is {\em continuously Pansu differentiable} if the map $x \mapsto z(x,\cdot)$ is continuous.

\subsection{Modulus of continuity}

\begin{Def} \label{d-modcont}
    Suppose $\mathcal{A} \subset \mathbb{G}$ is compact and fix $f:A \to \hat{\mathbb{G}}$. 
    We say that $\omega:[0,\infty) \to [0,\infty)$ is a {\em Euclidean modulus of continuity} on $\mathcal{A}$ of $f$ if it is non-constant, continuous, concave, increasing, $\omega(0)=0$, and
$$
|f(x) - f(y)| \leq \omega(d(x,y))
\quad \text{ for all } x,y \in \mathcal{A}.
$$
\end{Def}
Any continuous function on a compact set yields a concave modulus of continuity.
It follows from the concavity of $\omega$ that $t \mapsto \omega(t)/t$ is a decreasing function. See, for example, Lemma~2.5 in \cite{SpeZimCmw}.
From this, it follows that $\frac{\omega(t_A)}{t_A} \leq \frac{\omega(t)}{t}$.
That is, there is a constant $c > 0$ depending only on $A$ and $\mathbb{G}$ such that $t \leq c \omega(t)$ for any $t \in [0,t_A]$.
Moreover, if $C \geq 1$ is a constant, then, for any $t>0$,
$\frac{\omega(Ct)}{Ct} \leq \frac{\omega(t)}{t}$
so that $\omega(Ct) \leq C \omega(t)$.
These facts will be used throughout the proofs below.

\section{Approximating smooth, horizontal curves with horizontal lines}
\label{s3}

We make the following observation about horizontal curves.
Given a curve $\gamma:[a,b] \to \mathbb{G}$ that is horizontal at $t \in [a,b]$, 
write $h(t) := (\gamma_1'(t),\dots,\gamma_n'(t),0,\dots,0)$.
Note that $\alpha h(t) = \delta_{\alpha}(h(t))$.
\begin{lem}
    \label{p-possess}
    Suppose $\gamma:[a,b] \to \mathbb{G}$ is a curve which is horizontal at $t \in [a,b]$. 
    Then, for each $i \in \{ n+1 , \dots,N\}$ and $\alpha > 0$, we have that
    $
    \alpha^{d_i}\gamma_i'(t)  
    $
    is a finite sum of terms each of which possesses 
    $$
    \alpha^{d_j+1}\gamma_j(t) h_\ell(t) - \alpha^{d_\ell + 1}\gamma_\ell (t) h_j(t)
    $$
    as a factor for some $1 \leq j,\ell < i$.
\end{lem}

\begin{proof}
    Fix $i \in \{ n+1 , \dots,N\}$.
    According to \eqref{e-horizniceintro}, we have
    \begin{equation}
        \label{e-horiznice}
    \gamma_i'(t) = \sum_{j=1}^n \gamma_j'(t) \left. \frac{\partial Q_i(\gamma(t),y)}{\partial y_j}\right|_{y=0}
    =
    \bar{Q}_i(\gamma(t),h(t))
    \end{equation}
    where $\bar{Q}_i$ is the finite sum of the terms in $Q_i(x,y)$ in which $y$ appears linearly. Proposition~\ref{p-structure}{\em (2)} implies that
    $\bar{Q}_i(x,y)$ is a finite sum of terms of the form 
    $$
    p_{j,\ell}(x)(x_j y_\ell - x_\ell y_j )
    $$
    for $1 \leq j,\ell < i$ where $p_{j,\ell}$ is a polynomial.
    The $d_i$-degree homogeneity of $Q_i$ (and thus of $\bar{Q}_i$) gives
    \begin{align*}
    \alpha^{d_i} \gamma_i'(t) 
    = 
    \bar{Q}_i(\delta_\alpha (\gamma(t)),\delta_\alpha (h(t)))
    =
    \bar{Q}_i(\delta_\alpha (\gamma(t)),\alpha h(t)),
    \end{align*}
    and this is then a finite sum of terms of the form
    $$
    p_{j,\ell}\left(\delta_\alpha (\gamma(t))\right)\left(\alpha^{d_j+1}\gamma_j(t) h_\ell(t) - \alpha^{d_\ell + 1}\gamma_\ell (t) h_j(t) \right).
    $$
\end{proof}


Suppose $f:[0,b] \to \mathbb{G}$ is a horizontal curve such that $f_i$ is Euclidean-differentiable at 0 for $i \leq n$.
Define the ``horizontal ray'' 
$$
L_f(t) := f(0) * \left( tf_1'(0),\dots,tf_n'(0),0\dots,0 \right)
\quad
\text{for } t \geq 0.
$$

The following is a generalization of Lemmas 3 and 4 and Corollary 5 
from \cite{Warhurst}.
This is the primary new tool in the proof of Theorem~\ref{e-main1} and is similar to approximations proven in Proposition~21 from \cite{CLZC1a} and Lemma~3.9 from \cite{ZimBanachDual}.
However, unlike in these cited results, we now no longer assume regularity of every component of the curve but only its horizontal ones. It turns out that this lower level of regularity is still enough to guarantee that the curve is well approximated by a horizontal ray.

\begin{lem}
    \label{l-approx}
    Suppose $t_0 \geq 0$ and $f:[0,t_0] \to \mathbb{G}$ is a horizontal curve such that $f_i$ is Euclidean-$C^{1}$ for $i \leq n$, and suppose that $|f_i'(t) - f_i'(0)| \leq \omega(t)$ for some Euclidean modulus of continuity $\omega$, all $t \in [0,t_0]$, and $i \leq n$. 
    Then there is a constant $C_1>0$ depending only on $\omega$, $\mathbb{G}$, $t_0$, and $|f_i'(0)|$ for $i \leq n$ such that, for every $t \in [0,t_0]$,
    \begin{equation}
        \label{e-closetoline}
        d(f(t),L_{f}(t)) \leq C_1 t \, \omega(t)^{1/s}.
    \end{equation}
    
\end{lem}

\begin{proof}
    Throughout this proof, we will write $x \lesssim y$ for quantities $x,y \geq 0$ to indicate that there is some constant $C \geq 1$ depending only on $\omega$, $\mathbb{G}$, $t_0$, and $|f_i'(0)|$ for $i \leq n$ such that $x \leq Cy$.
    

    Using the left invariance of the metric on $\mathbb{G}$, we may assume without loss of generality that
        $f(0) = 0$.
            Notice first that, for $i \in \{1,\dots,n\}$ and all $t \in [0,t_0]$,
    \begin{align}
        \label{e-goal1}
        |f_i(t) - tf_i'(0)| \leq
        \int_0^t |f_i'(\tau) - f_i'(0)| \, d\tau
        \leq t \, \omega(t).
    \end{align}
    Since $d(f(t),L_f(t)) = N_\infty( L_f(t)^{-1} * f(t))$,
    it remains to verify that 
    \begin{equation}
        \label{e-stepgoal}
        |f_i(t)| + |Q_i(L_f(t)^{-1},f(t))| \lesssim t^{d_i} \, \omega(t)
    \end{equation}
    for $i \in \{ n+1, \dots, N\}$.

    Fix such an index $i$, and choose $t \in [0,t_0]$ at which $f$ is horizontal.
    Suppose by way of induction that we have proven for all $n+1 \leq j < i$ that 
    \begin{equation}
        \label{e-indstep}
        \left| f_j(t) \right|
        \lesssim t^{d_j} \, \omega(t).
    \end{equation}
    According to Lemma~\ref{p-possess},
    we can write $t^{1-d_i}f_i'(t) = t \left((t^{-1})^{d_i}f_i'(t) \right)$ as a finite sum of terms each of which is a multiple of
    \begin{equation}
        \label{e-inductstep}
    t^{-d_j}f_j(t)h_\ell(t) - t^{-d_\ell}f_\ell (t) h_j(t)
    \end{equation}
    with $1 \leq j,\ell < i$.
    If $j > n$, we have $h_j(t) = 0$. In this case (or when $\ell > n$), 
    we may use \eqref{e-indstep} 
    and the fact that $|f_k'(t)| \leq |f_k'(0)| + \omega(t_0)$ for $k \leq n$
    to bound \eqref{e-inductstep} by a constant multiple of $\omega(t)$ with the constant depending only on $\omega$, $\mathbb{G}$, $t_0$, and $|f_k'(0)|$ for $k \leq n$.
    Otherwise, if $j,\ell \leq n$, we have $d_j=d_\ell=1$, so \eqref{e-inductstep} is bounded by
    \begin{align*}
        \tfrac{1}{t} \left| f_j(t) - tf_j'(0)\right| \left|f_\ell' (t)\right|
        +
        \tfrac{1}{t} \left| f_\ell(t) - t f_\ell'(0)\right|\left|f_j' (t)\right|
        +
        \left| f_j'(0) f_\ell'(t) - f_\ell '(0) f_j'(t)\right|.
    \end{align*}
    This last term is bounded by
    \begin{align*}
        \left| f_j'(0)\right|
        \left| f_\ell '(t) - f_\ell '(0)\right|
        +
        \left| f_\ell '(0)\right|
        \left| f_j '(t) - f_j '(0)\right|
        \lesssim \omega(t).
    \end{align*}
    This, \eqref{e-goal1}, and the argument above imply that
    $$
    \left| t^{-d_j}f_j(t)h_\ell(t) - t^{-d_\ell}f_\ell (t) h_j(t) \right|
        \lesssim
    \omega(t)
    $$
    for any $1 \leq j,\ell < i$,
    and hence $\left|t^{1-d_i}f_i'(t)\right| \lesssim
    \omega(t)$ for almost all $t \in [0,t_0]$.
    Therefore, 
    for all $t \in [0,t_0]$ and $i \in \{n+1, \dots ,N\}$,
    \begin{align}
    \label{e-f-bound}
        \left|f_i(t)\right|
        \leq
        t^{d_i - 1} \int_0^t \left|t^{1-d_i} f_i'(t) \right| \, dt
        \lesssim t^{d_i} \, \omega(t).
    \end{align}

    It remains to bound the second term in \eqref{e-stepgoal}, and the arguments here are similar.
    Again, choose $t \in [0,t_0]$ at which $f'(t)$ is horizontal and fix $i \in \{n+1,\dots,N\}$.
    Proposition~\ref{p-structure} implies that 
    $$
    t^{-d_i}Q_i\left(L_f(t)^{-1},f(t)\right) = Q_i\left(\delta_{1/t}(L_f(t)^{-1}),\delta_{1/t}(f(t))\right),
    $$
    and this polynomial is a sum of terms each of which contains a factor of the form
    $$
    t^{-d_j}f_j(t)h_\ell(0) - t^{-d_\ell}f_\ell (t) h_j(0)
    $$
    for $1 \leq j ,\ell < i$.
    This quantity is bounded from above by
        \begin{align*}
        t^{-d_j}\left| f_j(t) - t^{d_j}h_j(0)\right| \left|t^{-d_\ell}f_\ell (t)\right|        +
        t^{-d_\ell}\left| f_\ell(t) - t^{d_\ell}h_\ell(0)\right| \left|t^{-d_j}f_j (t)\right|
    \end{align*}
    Note that \eqref{e-goal1} implies $\left| f_j(t) \right| \leq t \left(\omega(1)+\left|f_j'(0)\right|\right)$ when $j \leq n$.
    This together with \eqref{e-f-bound} gives
    $$
    \left| Q_i\left(L_f(t)^{-1},f(t)\right) \right| 
    \lesssim
    t^{d_i} \, \omega(t) 
    $$
    which completes estimate \eqref{e-stepgoal} and hence the proof.
\end{proof}

\section{Proof of the main theorem} \label{s4}

As in Remark~\ref{r-strat}, there are constants $k_0$ and $C_0$ depending only on $\mathbb{G}$ such that,
for any $\xi \in K := \partial B(0,1)$,
we can write $\xi = \xi_{k_0} \cdots \xi_{1}$ 
where each $\xi_k = \lambda_k e_{j_k}$ for some $j_k \in \{ 1,\dots,n\}$ and $\lambda_k \in \mathbb{R}$ with $|\lambda_k| \leq C_0 |\xi|$.

For any such $\xi$ and any $x \in \mathbb{G}$, we define
\begin{align*}
    D_{\xi_k} f(x) = \left( \left. \tfrac{d}{d\tau} f_1(x \delta_\tau(\xi_k)) \right|_{\tau=0},\dots,\left. \tfrac{d}{d\tau} f_{\hat{n}}(x \delta_\tau(\xi_k)) \right|_{\tau=0},0,\dots,0 \right)
\end{align*}
and
$$
z(x,\xi) = D_{\xi_{k_0}}f(x) D_{\xi_{k_0 -1}}f(x) \cdots D_{\xi_1}f(x).
$$
Recall also that, given $A \subset \Omega \subset \mathbb{G}$, we define $t_A = \frac{\text{dist}(A, \mathbb{G} \setminus \Omega)}{2k_0 C_0\sup_{ \Vert \xi \Vert = 1} |\xi|}$.

We restate the quantitative version of Theorem~\ref{t-actualmain} here for convenience.
In the following, we will use $d$, $\delta$, $N$, and $n$ and $\hat{d}$, $\hat{\delta}$, $\hat{N}$, and $\hat{n}$ to denote the metrics on, dilations on, topological dimensions of, and horizontal dimensions of $\mathbb{G}$ and $\hat{\mathbb{G}}$ respectively.
\begin{thm}
    Suppose $\Omega \subset \mathbb{G}$ is open. If $f: \Omega \to \hat{\mathbb{G}}$ preserves horizontal curves and $f_i$ is Euclidean-$C^1$ for $i \leq \hat{n}$, then,
    for any compact set $A \subset \Omega$, there is a constant $C > 0$ satisfying the following:
    for every $x \in A$ and $\xi \in \partial B(0,1)$,
    \begin{equation}
        \label{e-main3}
        \hat{d}(R(x,\xi;t),z(x,\xi)) \leq C\omega(t)^{1/s^{k_0}}
    \end{equation}
    for $t \in \left(0,t_A\right]$.
\end{thm}

\begin{proof}
    Throughout this proof, we will write $a \lesssim b$ for quantities $a,b \geq 0$ to indicate that there is some constant $C \geq 1$ depending only on $\mathbb{G}$, $A$, $\Omega$, and $f$ such that $a \leq Cb$.

We begin by verifying the equality in \eqref{e-otherform}.
Note first that, for any $x \in \mathbb{G}$,
\begin{align*}
    \left. \tfrac{d}{d\tau} x \delta_\tau(\xi_k) \right|_{\tau=0}
    =
     \left. \tfrac{d}{d\tau} L_x (\tau \lambda_k e_{j_k})
    \right|_{\tau=0}
    =
     d L_x \left( \lambda_k \partial_{j_k}    \right)
     =
     \lambda_k X_{j_k}(x)
\end{align*}
where $L_x(p) := xp$.
Thus,
given $f: \mathbb{G} \supset \Omega\to \hat{\mathbb{G}}$ such that $f_i$ is Euclidean-differentiable at $x \in \Omega$ for $i \leq \hat{n}$, we have
\begin{align*}
    D_{\xi_i} f(x) =\lambda_i \left( \nabla f_1(x) \cdot X_{j_i}(x), \dots, \nabla f_{\hat{n}}(x) \cdot X_{j_i}(x),0,\dots,0 \right).
\end{align*}

Now fix $f: \mathbb{G} \supset \Omega\to \hat{\mathbb{G}}$ and assume that $f_i$ is Euclidean-$C^1$ for $i \leq \hat{n}$.
 Fix a compact set $A \subset \Omega$ and choose a function $\omega$ which is a Euclidean modulus of continuity on $\mathcal{A}$ for each $\nabla f_i$ with $i \leq \hat{n}$.
To justify our definition of $z(x,\xi)$, we first recall ``Pansu's trick'' for the difference quotient \cite{Pansu}:
given $x \in A$, $p,q \in K$, and $t \in [0,t_A]$, we have
\begin{align*}
    R(x,pq;t) &= 
    \hat{\delta}_{1/t} \left( f(x)^{-1} f\left( x\delta_{t}(pq) \right) \right)\\
    &= 
    \hat{\delta}_{1/t} \left( f(x)^{-1} f(x \delta_t(p))\right) \hat{\delta}_{1/t} \left(f(x \delta_t(p))^{-1}f\left( x\delta_{t}(p) \delta_{t}(q)\right) \right)\\
    &= R(x,p;t)R(x \delta_t(p),q;t).
\end{align*}
Applying this trick $k_0-1$ times, we have for any $x \in A$ that
\begin{align}
    R(x,\xi;t) 
    =
    R(x,\xi_{k_0} \cdots \xi_1;t) 
    &= R(x,\xi_{k_0};t)R(x \delta_t(\xi_{k_0}),\xi_{k_0-1}\cdots \xi_1;t) \nonumber \\
    &= \dots \nonumber \\
    &= R(x_{k_0},\xi_{k_0};t) R(x_{k_0-1}(t),\xi_{k_0-1};t) \cdots R(x_1(t),\xi_1;t) \label{e-deconstruct}
\end{align}
where $x_{k_0}(t) \equiv x$ and 
\begin{align*}
x_k(t) 
:= 
x_{k+1} \delta_t(\xi_{k+1})
=x \delta_t(\xi_{k_0}\xi_{k_0-1}\cdots\xi_{k+1})
\end{align*}
for $k = 1, \dots, {k_0}-1$.
The left invariance and homogeneity of $d$ together with the triangle inequality imply that $x_k(t) \in \mathcal{A} \subset \Omega$ for each $t \in [0,t_A]$.
Indeed,
\begin{align}
    d(x,x_k(t))
    =
    td(0,\xi_{k_0} \cdots \xi_{k+1})
    &\leq
    t(d(0,\xi_{k_0}) + \cdots + d(0,\xi_{k+1})) \nonumber \\
    &\leq t ( |\lambda_{k_0}| + \cdots | \lambda_{k+1}|) 
    \leq t(k_0-1)C_0|\xi|.  \label{e-insideOmega}
\end{align}

For each $t \in [0,t_A]$,
define 
$$
z_0(x,\xi;t) := D_{k_0}f(x) * D_{k_0 -1}f(x_{k_0-1} (t)) * \cdots * D_{1}f(x_1 (t)).
$$

We will complete the proof by verifying that 
\begin{equation}
    \label{e-finalgoal}
    \hat{d}(z(x,\xi),z_0(x,\xi;t)) \lesssim \omega(t)^{1/s^{k_0}}
\end{equation}
and
\begin{equation}
    \label{e-finalgoal2}
    \hat{d}(R(x,\xi;t),z_0(x,\xi;t)) \lesssim \omega(t)^{1/s^{k_0}}
\end{equation}
for all $x \in A$, $\xi \in K$, and $t \in [0,t_A]$.

We will first attack estimate \eqref{e-finalgoal}.
In view of Corollary~\ref{c-nice},
we will prove the following claim.

\noindent \textbf{Claim:}
There is a constant $C_1>0$ such that, for all $x \in A$ and $k \in \{1,\dots,k_0\}$,  
\begin{align}
    \label{e-approxinproof0}    
        \hat{d}(D_kf(x),D_kf(x_k(t))) \leq C_1 \omega(t)^{1/s}
\end{align}
for $t \in [0,t_A]$.
\begin{proof}[Proof of Claim]
By \eqref{e-insideOmega}, we have
$x_k(t) \in \mathcal{A}$, and so for $i \leq \hat{n}$ we get
\begin{align*}
&\left| \left.\tfrac{d}{d\tau} f_i(x \delta_\tau(\xi_k))\right|_{\tau=0} - \left.\tfrac{d}{d\tau} f_i(x_k(t)  \delta_\tau(\xi_k))\right|_{\tau=0} \right|\\
& \qquad  \qquad \qquad
=
| \nabla f_i(x) \cdot  
\lambda_k X_{j_k}(x)
-
\nabla f_i(x_k(t)) \cdot \lambda_k X_{j_k}(x_k(t))|\\
& \qquad  \qquad \qquad
\lesssim
|\nabla f_i(x)- \nabla f_i(x_k(t))| |X_{j_k}(x)|
+ |\nabla f_i(x)||X_{j_k}(x) - X_{j_k}(x_k(t))|\\
& \qquad  \qquad \qquad
\lesssim
\omega(t) |X_{j_k}(x)|
+ t |\nabla f_i(x)|.
\end{align*}
In the last inequality, we applied Proposition~\ref{p-compact} and \eqref{e-insideOmega}, and we utilized the smoothness of $X_{j_k}$.
Therefore, according to Proposition~\ref{p-compact},
$$
\hat{d}(D_kf(x),D_kf(x_k(t)))
\lesssim
|D_kf(x) - D_kf(x_k(t))|^{1/s} \lesssim \omega(t)^{1/s}.
$$
\end{proof}

Now, for $k \in \{1,\dots,k_0\}$, write $a_k = D_k f(x)$ and $b_k = D_k f(x_k(t))$.
Note that each $|a_k|$ and each $|b_k|$ is bounded by a constant depending only on $A$, $f$, $\Omega$, and $\mathbb{G}$ according to $\eqref{e-otherform}$ and the fact that $x_k(t) \in \mathcal{A}$.
Hence, the above claim applied to Corollary~\ref{c-nice} gives the bound
\begin{align*}
    \hat{d} \left( z(x,\xi),z_0(x,\xi;t) \right)
    =
    \hat{d}(a_{k_0} \cdots a_1,b_{k_0} \cdots b_1)
    \lesssim
    \omega(t)^{1/s^{k_0}}.
\end{align*}
This is estimate \eqref{e-finalgoal}.

It remains to bound $\hat{d}(R(x,\xi;t),z_0(x,\xi;t))$ for any  $x \in A$ and $\xi \in K$.
Recall that 
$t \mapsto x \delta_t(\xi_k)$ is a horizontal $C^\infty$ curve in $\mathbb{G}$ for each fixed $x \in \mathbb{G}$.
Hence
our assumptions on $f$ imply that the curve
$\eta:[0,t_A] \to \hat{\mathbb{G}}$ defined as
$\eta(t) = f(x \delta_t(\xi_k))$ is a horizontal curve in $\hat{\mathbb{G}}$, and $\eta_i$ is Euclidean-$C^1$ for $i \leq \hat{n}$.
We would like to apply Lemma~\ref{l-approx} to this curve.
Following our prior notation,
denote the ``horizontal ray'' in $\hat{\mathbb{G}}$ approximating $\eta$ by
\begin{align*}
L_{x,k}(t) := L_\eta(t) &= \eta(0) * \left( t \eta_1'(0),\dots,t \eta_{\hat{n}}'(0),0,\dots,0\right)\\
&= 
f(x) * \left( t\left. \tfrac{d}{d\tau} f_1(x \delta_\tau(\xi_k)) \right|_{\tau=0},\dots,t\left. \tfrac{d}{d\tau} f_{\hat{n}}(x \delta_\tau(\xi_k)) \right|_{\tau=0},0,\dots,0 \right)\\
&= f(x) * \hat{\delta}_t (D_kf(x)).
\end{align*}
We then have for any $x \in A$ that 
\begin{align}
    \hat{d} \left( R(x,\xi_k;t),D_kf(x)\right)
    &=
    \hat{d} \left( \hat{\delta}_{1/t} \left( f(x)^{-1} f\left( x\delta_{t}(\xi_k) \right) \right), 
    \hat{\delta}_{1/t} \left(f(x)^{-1} L_{x,k}(t) \right)\right) \label{e-abbound}\\
    &=
    \tfrac{1}{t} \hat{d}  \left(  f\left( x\delta_{t}(\xi_k) \right) , 
    L_{x,k}(t)  \right). \nonumber
\end{align}

\noindent \textbf{Claim:}
There is a constant $C_2>0$ depending only on $f$, $A$, $\Omega$, and $\mathbb{G}$ such that, for all $x \in A$, $k \in \{1,\dots,k_0\}$, and $t \in [0,t_A]$, 
\begin{align}
    \label{e-approxinproof}    
        \hat{d}(f(x_k(t) \delta_t(\xi_k)),L_{x_k(t),k}(t)) 
        \leq C_2 t \, \omega(t)^{1/s}.
\end{align}
Before proving this claim, let us see how it can be used to complete the proof of the theorem.
Estimates \eqref{e-abbound} and \eqref{e-approxinproof} together give
\begin{align}
\label{e-secondbound}
    \hat{d} \left( R(x_k(t),\xi_k;t),D_kf(x_k(t))\right) \leq C_2 \omega(t)^{1/s}
\end{align}
for all $x \in A$, $k \in \{1,\dots,k_0\}$, and $t \in [0,t_A]$.
We will now use Corollary~\ref{c-nice} 
to conclude \eqref{e-finalgoal2}.
Fix $x \in A$,
and, for each $k \in \{1,\dots,k_0\}$, write $a_k = R(x_k(t),\xi_k;t)$ and $b_k = D_k f(x_k(t))$.
We may bound $|b_k|$ as before, and \eqref{e-secondbound} then provides a bound on $|a_k|$ depending only on $f$, $A$, $\Omega$, and $\mathbb{G}$.
Therefore, we can apply \eqref{e-secondbound} and Corollary~\ref{c-nice} to conclude that
\begin{align*}
    \hat{d} \left( R(x,\xi;t),z_0(x,\xi;t) \right)
    =
    \hat{d}(a_{k_0} \cdots a_1,b_{k_0} \cdots b_1)
    \lesssim
    \omega(t)^{1/s^{k_0}}.
\end{align*}
This is the required bound to complete the proof.

\begin{proof}[Proof of Claim]
Fix $t \in [0,t_A]$.
Write $z= x_k(t)$, and set $\eta(\kappa) = f(z \delta_\kappa(\xi_k))$ for $\kappa \in [0,t_A]$.
We want to apply Lemma~\ref{l-approx}. In order to do so, it remains to prove that 
\begin{equation}
    \label{e-claimgoal}
    |\eta_i'(\kappa) - \eta_i'(0)| \lesssim \omega(\kappa)
\end{equation}
for all $\kappa \in [0,t_A]$ and $i \leq \hat{n}$.
As before, for such $i$,
\begin{align*}
    \eta_i'(\kappa)
    =
    \left. \tfrac{d}{d\tau} f_i(z \delta_\tau(\xi_k)) \right|_{\tau=\kappa}
    =
     \nabla f_i(z \delta_{\kappa}(\xi_k)) \cdot X_{j_k}(z)
\end{align*}
so that
\begin{align*}
    |\eta_i'(\kappa)-\eta_i'(0)| 
    &=
    \left| \nabla f_i(z \delta_{\kappa}(\xi_k)) \cdot X_{j_k}(z)
    -
    \nabla f_i(z) \cdot X_{j_k}(z)
    \right|\\
    &\leq 
\left| \nabla f_i(z \delta_{\kappa}(\xi_k))
    -
    \nabla f_i(z)
    \right|
    |X_{j_k}(z)|.
\end{align*}
Note that $z \delta_{\kappa}(\xi_k) \in \mathcal{A}$, so we have \eqref{e-claimgoal} by the definition of $\omega$.
Applying Lemma~\ref{l-approx} proves 
$$
\hat{d}(f(x_k(t) \delta_\kappa(\xi_k)),L_{x_k(t),k}(\kappa)) 
        \lesssim \kappa \, \omega(\kappa)^{1/s}
$$
for all $\kappa \in [0,t_A]$. Choosing specifically $\kappa = t$ completes the proof of the claim.
\end{proof}

This establishes \eqref{e-finalgoal2} and proves the theorem.
\end{proof}

\bibliographystyle{alpha}
\bibliography{bibliography}
\end{document}